\documentclass[letterpaper, 10 pt, conference]{ieeeconf}   

\IEEEoverridecommandlockouts                   

\usepackage{graphics}
\usepackage{epsfig} 
\usepackage{times}
\usepackage{amsmath}
\usepackage{amssymb}
\usepackage{dsfont}
\usepackage{enumerate}
\usepackage[tight]{subfigure}
\usepackage{multirow}
\usepackage{bm}
\usepackage{tikz}
\usepackage[noadjust]{cite}

\usepackage{hyperref}

\usetikzlibrary{arrows.meta}
\usetikzlibrary{positioning}

\usetikzlibrary{shadows,shadings,shapes.symbols, calc}

\newcommand{\F}{\mathcal{F}}

\newcommand{\Ftwo}{\mathcal{F}_{Z,NL}}

\newcommand{\Nme}{N^{m,e}}
\newcommand{\NN}{\mathcal{N}}
\newcommand{\vv}{\mathbf{v}}

\newcommand{\FF}{\mathfrak{F}}

\newcommand{\ww}{\mathfrak{w}}

\usepackage{color}

\title{Nonlinear identifiability of directed acyclic graphs \linebreak with partial excitation and measurement}

\author{Renato Vizuete and Julien M. Hendrickx
\thanks{*This work was supported by F.R.S.-FNRS via the \emph{KORNET} project and via the Incentive Grant for Scientific Research (MIS) \emph{Learning from Pairwise Comparisons}, and by the \emph{RevealFlight} Concerted Research Action (ARC) of the Fédération Wallonie-Bruxelles.}
\thanks{R.~Vizuete and J.~M.~Hendrickx are with ICTEAM institute, UCLouvain, B-1348, Louvain-la-Neuve, Belgium. R.~Vizuete is a FNRS Postdoctoral Researcher - CR.
{\tt\small renato.vizueteharo@uclouvain.be},
{\tt\small julien.hendrickx@uclouvain.be}\protect.}}

\newcommand{\vertiii}[1]{{\left\vert\kern-0.25ex\left\vert\kern-0.25ex\left\vert #1 
    \right\vert\kern-0.25ex\right\vert\kern-0.25ex\right\vert}}

\newtheorem{definition}{Definition}
\newtheorem{theorem}{Theorem}

\newtheorem{proposition}{Proposition}
\newtheorem{lemma}{Lemma}
\newtheorem{remark}{Remark}
\newtheorem{assumption}{Assumption}
\newtheorem{example}{Example}

\newcommand{\abs}[1]{\left|#1\right|}

\newcommand{\R}{\mathbb R}

\begin{document}

\maketitle
\thispagestyle{empty}

\begin{abstract}
We analyze the identifiability of directed acyclic graphs in the case of partial excitation and measurement. We consider an additive model where the nonlinear functions located in the edges depend only on a past input, and we analyze the identifiability problem in the class of pure nonlinear functions satisfying $f(0)=0$. We show that any identification pattern (set of measured nodes and set of excited nodes) requires the excitation of sources, measurement of sinks and the excitation or measurement of the other nodes. Then, we show that a directed acyclic graph (DAG) is identifiable with a given identification pattern if and only if it is identifiable with the measurement of all the nodes. Next, we analyze the case of trees where we prove that any identification pattern guarantees the identifiability of the network. Finally, by introducing the notion of a generic nonlinear network matrix, we provide sufficient conditions for the identifiability of DAGs based on the notion of vertex-disjoint paths.
\end{abstract}

\section{Introduction}

Systems composed by single entities or units that interact in a network to generate a more complex global behavior are ubiquitous \cite{boccaletti2006complex,bullo2022lectures}. Many of these interactions are characterized by dynamics located at the level of edges, and the identification of these dynamics is essential to study the evolution, analyze the stability, and design control actions. Since the size of the networks can be really large and the design of experiments generally involves an economic cost, it is important to derive identifiability conditions that specify which nodes must be excited and measured to guarantee the identification of the network.
While the identifiability of linear interactions modeled by transfer functions have been extensively analyzed in recent years \cite{weerts2018identifiability,hendrickx2019identifiability,vanwaarde2020necessary}, few works have focused in nonlinear networks due to the complexity of the dynamics. The case of nonlinear identifiability where all the nodes can be excited (full excitation) has been analyzed in \cite{vizuete2023nonlinear,vizuete2024nonlinear}, where surprisingly, the identifiability conditions are weaker than in the linear case when pure nonlinear functions are considered. These differences between the linear and nonlinear cases can be explained by the consequences of the superposition principle that allows the mixing of certain information only in the linear case.

In many scenarios, due to physical limitations, some nodes might not be available to be excited, or the cost associated with the excitation of a node (actuator) might be greater than the cost associated with the measurement of a node (sensor). For instance, in an electrical network, the excitation of a node might require a power source while the measurement of a node might require only a sensor. 
Therefore, it is important to relax the full excitation assumption and guarantee at the same time the identifiability of all the network for more complex identification patterns (combination of excited nodes and measured nodes).

The identifiability of linear networks in the case of partial excitation and measurement has been well analyzed in several works \cite{bazanella2019network,legat2020local,cheng2023necessary}, where the symmetry excitation/measurement plays an important role. Nevertheless, even in the linear case, many problems are still open such as the minimum number of nodes that must be excited and measured, a characterization of the identification patterns that guarantee identifiability of the network, among others.

In this work, we analyze the identifiability of nonlinear networks in the case of partial excitation and measurement. We consider a model where the dynamics depend only on a past input, and we restrict the analysis to the case of directed acyclic graphs (DAGs) and pure nonlinear surjective functions. 
We show that any identification pattern requires the excitation of sources, the measurement of sinks, and the excitation or measurement of the other nodes of the network. Then, we show that a DAG is identifiable with an identification pattern if and only if it is identifiable with the measurement of all the nodes, which is different from the linear case. In the case of trees, we show that there is a symmetry between the identifiability conditions for the full excitation and full measurement cases and that any identification pattern guarantees identifiability. Unlike the linear case, for more general DAGs, we show that the symmetry does not hold, and we derive sufficient conditions for identifiability based on the notion of vertex-disjoint paths and a generic nonlinear network matrix.

\section{Problem formulation}

\subsection{Model class}

We consider a network characterized by a weakly connected digraph $G=(V,E)$ composed by a set of nodes $V=\{ 1,\ldots,n\}$ and a set of edges $E\subseteq V\times V$. The output of each node $i$ in the network is given by:
\begin{equation}\label{eq:nonlinear_model}
y_i^k=\sum_{j\in \mathcal{N}_i}f_{i,j}(y_j^{k-{m_{i,j}}})+u_i^{k-1}, \;\; \text{for all  } i\in V,  
\end{equation}
where the superscripts of the inputs and outputs  denote the corresponding values at the specific time instants, the delay $m_{i,j}\in\mathbb{Z}^+$ is finite, $f_{i,j}$ is a nonlinear function, $\mathcal{N}_i$ is the set of in-neighbors of node $i$, and $u_i$ is an arbitrary external excitation signal. If a node is not excited, its corresponding excitation signal is set to zero. Fig.~\ref{fig:model} presents an illustration of the model considered in this work. 
Our goal is to analyze a general nonlinear dynamics, but in this preliminary work, we will focus on the effect of the nonlinearities. The model  \eqref{eq:nonlinear_model} corresponds to a generalized version of the nonlinear static model in \cite{vizuete2023nonlinear}, where $m_{i,j}$ can take a value different from 1. 

The nonzero functions $f_{i,j}$ between the nodes define the topology of the network $G$, forming the set of edges $E$. We do not consider multiple edges between two nodes since they would be indistinguishable and hence unidentifiable.

\begin{assumption}\label{ass:full_excitation}
    The topology of the network is known, where the presence of an edge implies a nonzero function.
\end{assumption}

Assumption~\ref{ass:full_excitation} implies that we know which nodes are connected by nonzero functions. 
The objective of this work is to determine which nodes need to be excited and/or measured to identify all the nonlinear functions in the network. Our aim is to determine the possibility of identification and not to develop an algorithm or verify the accuracy of other identification methods.

Similarly to \cite{hendrickx2019identifiability,vizuete2023nonlinear}, for the identification process we assume that the relations between the signals of excited nodes and the outputs of measured nodes have been perfectly identified. In addition, we restrict our attention to networks that do not contain any cycle (i.e., directed acyclic graphs). This implies that when we measure a node $i$, we obtain an identification of the function $F_i$:
\begin{multline}\label{eq:function_Fi}
  \!\!\!\!\!  y_i^k=u_i^{k-1}+F_i(u_1^{k-2},\ldots,u_1^{k-M_1},\ldots,u_{n_i}^{k-2},\ldots,u_{n_i}^{k-M_{n_i}}),\\
    1,\dots,n_i\in \mathcal{N}^{e\to i},
\end{multline}
where $\mathcal{N}^{e\to i}$ is the set of excited nodes with a path to the node $i$. The function $F_i$ determines the output of the node $i$ based on the excitation signals and only depends on a finite number of inputs (i.e., $M_1,\ldots,M_{n_i}$ are finite) due to the finite delays $m_{i,j}$ and the absence of cycles.
With a slight abuse of notation, we use the superscript in the function $F_i^{(s)}$ to denote that all the inputs in \eqref{eq:function_Fi} are delayed by $s$:

\vspace{-3mm}

\small
$$
F_i^{(s)}=F_i(u_1^{k-2-s},\ldots,u_1^{k-M_1-s},\ldots,u_{n_i}^{k-2-s},\ldots,u_{n_i}^{k-M_{n_i}-s}).
$$

\normalsize

\begin{figure}[!t]
    \centering
    \begin{tikzpicture}
    [fill fraction/.style n args={2}{path picture={
            \fill[#1] (path picture bounding box.south west) rectangle
            ($(path picture bounding box.north west)!#2!(path picture bounding box.north east)$);}},
roundnodes/.style={circle, draw=black!60, fill=black!5, very thick, minimum size=1mm},roundnode/.style={circle, draw=white!60, fill=white!5, very thick, minimum size=1mm},roundnodes2/.style={circle, draw=black!60, fill=black!30, very thick, minimum size=1mm},roundnodes3/.style={circle, draw=black!60, fill=white!30, very thick, minimum size=1mm},roundnodes4/.style={circle, draw=black!60, fill=black!30, fill fraction={white!30}{0.5}, very thick, minimum size=1mm}
]

\small

\node[roundnodes3](node1){1};
\node[roundnodes2](node2)[above=of node1,yshift=3mm]{2};
\node[roundnodes2](node3)[right=1.8cm of node1]{3};
\node[roundnodes4](node4)[right=1.8cm of node2]{4};
\node[roundnodes3](node5)[right=1.8cm of node3]{5};
\node[roundnodes4](node6)[right=1.8cm of node4]{6};
\node[roundnode](u1)[above=of node1,yshift=-7mm,xshift=-8mm]{$u_1$};
\node[roundnode](u4)[above=of node4,yshift=-7mm,xshift=-8mm]{$u_4$};
\node[roundnode](u6)[above=of node6,yshift=-7mm,xshift=8mm]{$u_6$};
\node[roundnode](u5)[above=of node5,yshift=-7mm,xshift=8mm]{$u_5$};

\draw[-{Classical TikZ Rightarrow[length=1.5mm]}] (node1) to node [right,swap,yshift=0mm] {$f_{2,1}$} (node2);
\draw[-{Classical TikZ Rightarrow[length=1.5mm]}] (node2) to node [above,swap,yshift=0mm] {$f_{4,2}$} (node4);
\draw[-{Classical TikZ Rightarrow[length=1.5mm]}] (node1) to node [below,swap,yshift=0mm] {$f_{3,1}$} (node3);
\draw[-{Classical TikZ Rightarrow[length=1.5mm]}] (node2) to node [right,swap,yshift=2mm] {$f_{3,2}$} (node3);
\draw[-{Classical TikZ Rightarrow[length=1.5mm]}] (node4) to node [above,swap,yshift=0mm] {$f_{6,4}$} (node6);
\draw[-{Classical TikZ Rightarrow[length=1.5mm]}] (node3) to node [below,swap,yshift=0mm] {$f_{5,3}$} (node5);
\draw[-{Classical TikZ Rightarrow[length=1.5mm]}] (node5) to node [left,swap,yshift=0mm] {$f_{6,5}$} (node6);
\draw[-{Classical TikZ Rightarrow[length=1.5mm]}] (node3) to node [left,swap,yshift=2mm] {$f_{6,3}$} (node6);
\draw[gray,dashed,-{Classical TikZ Rightarrow[length=1.5mm]}] (u1) -- (node1);
\draw[gray,dashed,-{Classical TikZ Rightarrow[length=1.5mm]}] (u6) -- (node6);
\draw[gray,dashed,-{Classical TikZ Rightarrow[length=1.5mm]}] (u4) -- (node4);
\draw[gray,dashed,-{Classical TikZ Rightarrow[length=1.5mm]}] (u5) -- (node5);

\normalsize

\end{tikzpicture}

\vspace{-2mm}
    
    \caption{Model of a network considered for the identification where some nodes are excited (white) and some nodes are measured (gray). Other nodes are excited and measured at the same time (white and gray).}
    \vspace{-4mm}
    \label{fig:model}
\end{figure}
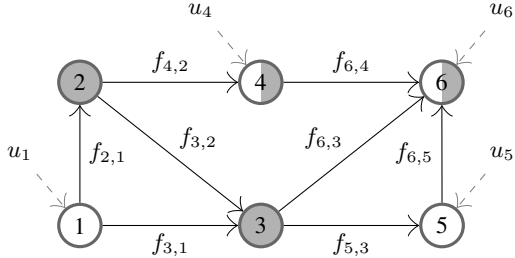

\subsection{Identifiability}

\begin{definition}[Identification pattern]
    We define an identification pattern as a pair $(\NN^e,\NN^m)$ where $\NN^e\subseteq V$ is the set of excited nodes and $\NN^m\subseteq V$ is the set of measured nodes.
\end{definition}

\begin{definition}
    We define the number of identification actions as $N^{m,e}=\abs{\NN^e}+\abs{\NN^m}$.
\end{definition}

Clearly, if a node $i$ is excited and measured at the same time (i.e., $i\in\NN^e$ and $i\in\NN^m$), this increases $N^{m,e}$ by 2.

Next, we define the relationships between the measurements of the nodes and the functions $f_{i,j}$.

\begin{definition}[Set of measured functions]\label{def:set_measured_functions}
    Given a set of measured nodes $\mathcal{N}^m$, the set of measured functions $F(\mathcal{N}^m)$ associated with $\mathcal{N}^m$ is given by:
    $$
    F(\mathcal{N}^m):=\{F_i\;|\;i\in \mathcal{N}^m\}.
    $$
\end{definition}

We say that a function $f_{i,j}$ associated with an edge satisfies $F(\mathcal{N}^m)$ if $f_{i,j}$ can lead to $F(\mathcal{N}^m)$ through \eqref{eq:nonlinear_model}.

Since the identifiability problem can be hard or unrealistic for general functions, we restrict the problem to  a certain class of functions $\F$: the functions associated with the edges belong to $\F$ and that the identifiability is considered only among the functions belonging to $\F$. 

\begin{definition}[Identifiability]\label{def:math_identifiability}
    Given a set of functions $\{ f \}=\{f_{i,j}\in \F\;|\;(i,j)\in E\}$ that generates $F(\NN^m)$ and another set of functions $\{ \tilde f \}=\{\tilde f_{i,j}\in\F\;|\;(i,j)\in E\}$ that generates $\tilde F(\NN^m)$. An edge $f_{i,j}$ is identifiable in a class $\F$ if $
    F(\NN^m)=\tilde F(\NN^m)
    $, implies that $f_{i,j}=\tilde f_{i,j}$.
    A network $G$ is identifiable in a class $\F$ if 
    $
    F(\NN^m)=\tilde F(\NN^m),
    $
    implies that
    $
    \{ f \}=\{ \tilde f \}.
    $
\end{definition}

In this work, we will consider analytic functions with a Taylor series that converges to the function for all $x\in\R$ (possibly with a finite radius of convergence).
In the full excitation case \cite{vizuete2023nonlinear,vizuete2024nonlinear}, it has been proved that a static component of a function could make the identifiability problem unsolvable. Also, when linear functions are allowed in a class $\F$, the identifiability conditions are similar to the linear case. However, it has been observed that the presence of nonlinearities change the identifiability conditions \cite{vizuete2023nonlinear,vizuete2024nonlinear}. For this reason, we will focus on the class of pure nonlinear functions \cite{vizuete2023nonlinear}.

\begin{definition}[Class of functions $\Ftwo$]\label{def:class_functions}
Let $\Ftwo$ be the class of functions $f:\R\to\R$ with the following properties:
\begin{enumerate}
    \item $f$ is analytic in $\R$.
    \item $f(0)=0$.
    \item The associated Taylor series $f(x)=\sum_{n=1}^\infty a_nx^n$ contains at least one coefficient $a_n\neq 0$ with $n>1$. 
    \item The range of $f$ is $\R$.
\end{enumerate}
\end{definition}

Therefore, in the rest of this work we will assume that $f_{i,j}\in\Ftwo$.

\section{Necessary conditions for identifiability}

Since the sources and sinks play an important role in directed acyclic graphs, we first provide necessary conditions that these nodes must satisfy to guarantee identifiability.

\begin{lemma}[Sinks and sources]\label{lemma:sinks_sources}
For identifiability of a DAG, it is necessary to measure all the sinks and excite all the sources.
\end{lemma}
\begin{proof}
    The necessity of the measurement of all the sinks has been proved in \cite[Proposition~1]{vizuete2023nonlinear}. For the necessity of the excitation of the sources, let us consider an arbitrary source $i$. 
    The measurement of an out-neighbor $j$ of the source $i$ provides the output:
    \begin{align}
    y_j^k&=\sum_{\ell\in\NN_j} f_{j,\ell}(y_\ell^{k-m_{j,\ell}})+u_j^{k-1}\nonumber\\
    &=f_{j,i}(u_i^{k-m_{j,i}-1})+\!\!\sum_{\ell\in\NN_j\setminus \{i\}} f_{j,\ell}(y_\ell^{k-m_{j,\ell}})+u_j^{k-1}.\label{eq:measurement_source}
    \end{align}
    Since the source $i$ is not excited, \eqref{eq:measurement_source} becomes:
    \begin{align}
    y_j^k&=f_{j,i}(0)+\sum_{\ell\in\NN_j\setminus \{i\}} f_{j,\ell}(y_\ell^{k-m_{j,\ell}})+u_j^{k-1}\nonumber\\
    &=\sum_{\ell\in\NN_j\setminus \{i\}} f_{j,\ell}(y_\ell^{k-m_{j,\ell}})+u_j^{k-1}.\label{eq:measurement_source_2}
    \end{align}
    Notice that any other function $\tilde f_{j,i}\in\Ftwo$ satisfies \eqref{eq:measurement_source_2}, which implies that it is not possible to identify the edge $f_{j,i}$. 
    Thus, the excitation of all the sources is necessary for identifiability of the network. 
\end{proof}

Next, we are interested in determining conditions that other nodes of the network must satisfy to guarantee identifiability.

\begin{proposition}\label{prop:excite_measure}
    For identifiability of a DAG, it is necessary to either excite or measure each node.
\end{proposition}
\begin{proof}
    By Lemma~\ref{lemma:sinks_sources}, the excitation of sources and the measurement of sinks is necessary.
    Let us consider an arbitrary node $i$ in a network which is neither a source nor a sink. Let us assume that we can excite or measure the other nodes and all the edges have been identified except the incoming and outgoing edges of $i$. The measurement of an out-neighbor $j$ of the node $i$ is given by:

\vspace{-2mm}
    
    \small
    \begin{align}
    y_j^k&\!=\!u_j^{k-1}\!+\!f_{j,i}(y_i^{k-m_{j,i}})\!+\!\sum_{p\in\NN_j\setminus \{i\}} f_{j,p}(y_{p}^{k-m_{j,p}})\nonumber\\
    &=\!u_j^{k-1}\!+\!f_{j,i}(\sum_{\ell\in\NN_i}\! f_{i,\ell}(y_{\ell}^{k-m_{j,i}-m_{i,\ell}}))\!+\!\!\!\!\sum_{p\in\NN_j\setminus \{i\}} \!f_{j,p}(y_{p}^{k-m_{j,p}}).\label{eq:measurement_any_node}
    \end{align}
    \normalsize
    Notice that the same output \eqref{eq:measurement_any_node} can be obtained with the functions $\tilde f_{i,\ell}(x)=\gamma f_{i,\ell}(x)$ and $\tilde f_{j,i}(x)=f_{j,i}(\frac{x}{\gamma})$, with $\gamma\neq 0$, which implies that the incoming edges and outgoing edges of $i$ cannot be identified. Thus, all the nodes must be either excited or measured.
\end{proof}

Similarly to the linear case \cite{bazanella2019network}, Proposition~\ref{prop:excite_measure} shows that we need at least $N^{m,e}=n$ to guarantee the identifiability of DAGs in $\Ftwo$. Nevertheless, unlike the linear case, where $N^{m,e}>n$ could be necessary for some DAGs \cite{mapurunga2022excitation}, in the nonlinear case, we have a stronger result.

\begin{proposition}
There exists an identification pattern with $\Nme=n$ that guarantees the identifiability of a DAG.
\end{proposition}
\begin{proof}
    It is a direct consequence of the case of full excitation where we can guarantee identifiability by measuring all the sinks and exciting all the other nodes \cite[Theorem 2]{vizuete2023nonlinear}.
\end{proof}

However, the full excitation case could be restrictive since in many situations, the excitation of all the nodes could not be possible, or the cost of exciting a node could be significantly higher than the cost of measuring a node. For this reason, we are interested in other identification patterns $(\NN^e,\NN^m)$ that guarantee identifiability of a DAG with only $\Nme=n$.

In the rest of this work, we will consider that in any identification pattern, all the sources are excited, all the sinks are measured, and the other nodes are either excited or measured.

\section{Full measurement equivalence}

In the next proposition, we will prove a link between any identification pattern and the full measurement case.

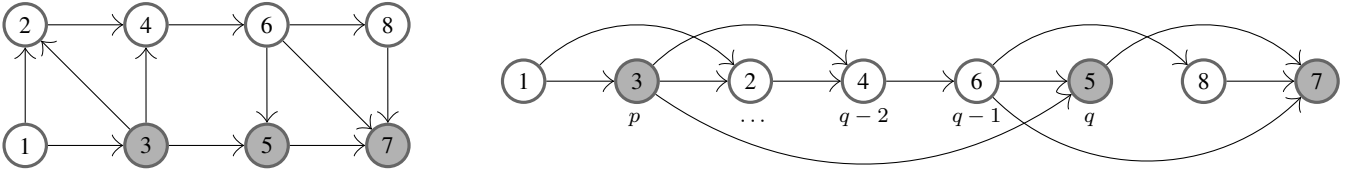
\begin{figure*}[!t]
    \centering

\vspace{2mm}
    
    \begin{tikzpicture}
    [fill fraction/.style n args={2}{path picture={
            \fill[#1] (path picture bounding box.south west) rectangle
            ($(path picture bounding box.north west)!#2!(path picture bounding box.north east)$);}},
roundnodes/.style={circle, draw=black!60, fill=black!5, very thick, minimum size=1mm},roundnode/.style={circle, draw=white!60, fill=white!5, very thick, minimum size=1mm},roundnodes2/.style={circle, draw=black!60, fill=black!30, very thick, minimum size=1mm},roundnodes3/.style={circle, draw=black!60, fill=white!30, very thick, minimum size=1mm},roundnodes4/.style={circle, draw=black!60, fill=black!30, fill fraction={white!30}{0.5}, very thick, minimum size=1mm}
]

\small

\node[roundnodes3](node1){1};
\node[roundnodes3](node2)[above=of node1,yshift=0mm]{2};
\node[roundnodes2](node3)[right=1.0cm of node1]{3};
\node[roundnodes3](node4)[right=1.0cm of node2]{4};
\node[roundnodes2](node5)[right=1.0cm of node3]{5};
\node[roundnodes3](node6)[right=1.0cm of node4]{6};
\node[roundnodes2](node7)[right=1.0cm of node5]{7};
\node[roundnodes3](node8)[right=1.0cm of node6]{8};

\node[roundnodes3](node9)[right=1.2cm of node8,yshift=-7.5mm]{1};
\node[roundnodes2](node10)[right=0.9cm of node9]{3} (8.1,0.15) node[above] {\footnotesize{$p$}};
\node[roundnodes3](node11)[right=0.9cm of node10]{2}(9.7,0.15) node[above] {\footnotesize{$\cdots$}};
\node[roundnodes3](node12)[right=0.9cm of node11]{4}(11.15,0.15) node[above] {\footnotesize{$q-2$}};
\node[roundnodes3](node13)[right=0.9cm of node12]{6}(12.65,0.15) node[above] {\footnotesize{$q-1$}};
\node[roundnodes2](node14)[right=0.9cm of node13]{5}(14.15,0.15) node[above] {\footnotesize{$q$}};
\node[roundnodes3](node15)[right=0.9cm of node14]{8};
\node[roundnodes2](node16)[right=0.9cm of node15]{7};

\draw[-{Classical TikZ Rightarrow[length=1.5mm]}] (node1) to node [right,swap,yshift=0mm] {} (node2);
\draw[-{Classical TikZ Rightarrow[length=1.5mm]}] (node2) to node [above,swap,yshift=0mm] {} (node4);
\draw[-{Classical TikZ Rightarrow[length=1.5mm]}] (node1) to node [below,swap,yshift=0mm] {} (node3);
\draw[-{Classical TikZ Rightarrow[length=1.5mm]}] (node4) to node [above,swap,yshift=0mm] {} (node6);
\draw[-{Classical TikZ Rightarrow[length=1.5mm]}] (node3) to node [below,swap,yshift=0mm] {} (node5);
\draw[-{Classical TikZ Rightarrow[length=1.5mm]}] (node6) to node [left,swap,yshift=0mm] {} (node5);
\draw[-{Classical TikZ Rightarrow[length=1.5mm]}] (node3) to node [left,swap,yshift=0mm] {} (node4);
\draw[-{Classical TikZ Rightarrow[length=1.5mm]}] (node6) to node [left,swap,yshift=0mm] {} (node8);
\draw[-{Classical TikZ Rightarrow[length=1.5mm]}] (node5) to node [left,swap,yshift=0mm] {} (node7);
\draw[-{Classical TikZ Rightarrow[length=1.5mm]}] (node8) to node [left,swap,yshift=0mm] {} (node7);
\draw[-{Classical TikZ Rightarrow[length=1.5mm]}] (node6) to node [left,swap,yshift=0mm] {} (node7);
\draw[-{Classical TikZ Rightarrow[length=1.5mm]}] (node3) to node [left,swap,yshift=0mm] {} (node2);

\draw[-{Classical TikZ Rightarrow[length=1.5mm]}] (node9) to node [left,swap,yshift=0mm] {} (node10);
\draw[-{Classical TikZ Rightarrow[length=1.5mm]}] (node11) to node [left,swap,yshift=0mm] {} (node12);
\draw[-{Classical TikZ Rightarrow[length=1.5mm]}] (node12) to node [left,swap,yshift=0mm] {} (node13);
\draw[-{Classical TikZ Rightarrow[length=1.5mm]}] (node13) to node [left,swap,yshift=0mm] {} (node14);
\draw[-{Classical TikZ Rightarrow[length=1.5mm]}] (node15) to node [left,swap,yshift=0mm] {} (node16);
\draw[-{Classical TikZ Rightarrow[length=1.5mm]}] (node10) to node [left,swap,yshift=0mm] {} (node11);

\draw[-{Classical TikZ Rightarrow[length=1.5mm]}] (node9) to[out=45, in=135, looseness=1] (node11);
\draw[-{Classical TikZ Rightarrow[length=1.5mm]}] (node10) to[out=45, in=135, looseness=1] (node12);
\draw[-{Classical TikZ Rightarrow[length=1.5mm]}] (node13) to[out=45, in=135, looseness=1] (node15);
\draw[-{Classical TikZ Rightarrow[length=1.5mm]}] (node14) to[out=45, in=135, looseness=1] (node16);
\draw[-{Classical TikZ Rightarrow[length=1.5mm]}] (node10) to[out=325, in=215, looseness=1] (node14);
\draw[-{Classical TikZ Rightarrow[length=1.5mm]}] (node13) to[out=315, in=225, looseness=1] (node16);

\normalsize

\end{tikzpicture}

\vspace{-3mm}
    
    \caption{DAG with a topological ordering where between two consecutive measured nodes $p$ and $q$, all the excited nodes have a path to $q$. 
    }
    \vspace{-4mm}
    \label{fig:topological_order}
\end{figure*}

\begin{proposition}[Full measurement information]\label{prop:full_measurement_equivalence}
A DAG is identifiable in the class $\Ftwo$ with the identification pattern $(\NN^e,\NN^m)$ if and only if it is identifiable with the identification pattern $(\NN^e,V)$.
\end{proposition}
Before presenting the proof of Proposition~\ref{prop:full_measurement_equivalence}, we introduce two lemmas.

\begin{lemma}[Lemma 4 \cite{vizuete2024nonlinear}]\label{lemma:removal_node}
    In the full excitation case, given a DAG $G$ and a node $j$ with only one path to its out-neighbor $i$. Let us assume that $j$ has been measured and the edge $f_{i,j}$ is identifiable. The edges $f_{i,\ell}$ and functions $F_\ell$ for $\ell\in\NN_i$, are identifiable in $G$ if they are identifiable in the induced subgraph \footnote{An induced subgraph $G_S$ of a digraph $G$ is a subgraph formed by a subset $S$ of the vertices of $G$ and all the edges of $G$ that have both endpoints in $S$.} $G_{V\setminus\{ j\}}$ with the measurement of $i$.
\end{lemma}

\begin{lemma}[Lemma~5 \cite{vizuete2024nonlinear}    ]\label{lemma:auxiliary}
Given three nonzero analytic functions $f:\R\to \R$ and $g,\tilde g:\R^p\to\R$. Let us assume that  $f$ is not linear
and they satisfy:
$$
f(0)= g(0)= \tilde g(0)=0;
$$
\begin{equation}    
\label{eq:statement_lemma_function_h}
f(x\!+\!g(y))=     
f(x\!+\!\tilde g(y))+h,
\end{equation}
where $h$ is an arbitrary function that does not depend on $x$.
Then $g=\tilde g$.
\end{lemma}

\begin{myproof}{Proposition~\ref{prop:full_measurement_equivalence}}
Clearly, if the DAG is identifiable with $(\NN^e,\NN^m)$, is also identifiable with $(\NN^e,V)$ since the full measurement case also includes the information of $\NN^m$ used for the identifiability.
    Now, let us consider an arbitrary DAG with a set of measured nodes $\NN^m$ and a set of excited nodes $\NN^e$ where all the sources are excited and all the sinks are measured. Since any DAG can be sorted in topological order \footnote{A topological ordering of a digraph is a linear ordering of its nodes such that for every edge $(i,j)$, $j$ comes before $i$ in the ordering.} \cite{bang2008digraphs}, let us consider a particular ordering where all the nodes located between two consecutive measured nodes $p$ and $q$, have a path to $q$. 
    Notice that if there is a topological ordering with a node $w$ between $p$ and $q$ that does not have a path to $q$, the node $w$ could be shifted to the right of $q$ in this topological ordering. 
    (For example, see Fig.~\ref{fig:topological_order}). The measurement of the node $q$ provides an output of the type:

    \vspace{-3mm}

    \small
    \begin{align*}
    y_q^k&=F_q\\
    &=\sum_{j\in\NN_q^e}f_{q,j}(u_j^{k-m_{q,j}-1}+F_j^{(m_{q,j})})\!+\!\sum_{\ell\in\NN_q^m}f_{q,\ell}(F_\ell^{(m_{q,\ell})}),    
    \end{align*}
    
    \normalsize
    \noindent    where $\NN_q^e$ is the set of excited in-neighbors of $q$ and $\NN_q^m$ is the set of measured in-neighbors of $q$. Since $F_q\in F(\NN^m)$, let us assume that there exists a set $\{\tilde f\}\neq \{ f\}$ such that $F_q=\tilde F_q$. Then, we obtain:
    \begin{multline*}
        \sum_{j\in\NN_q^e}f_{q,j}(u_j^{k-m_{q,j}-1}+F_j^{(m_{q,j})})+\sum_{\ell\in\NN_q^m}f_{q,\ell}(F_\ell^{(m_{q,\ell})})=\\
        \sum_{j\in\NN_q^e}\tilde f_{q,j}(u_j^{k-m_{q,j}-1}+\tilde F_j^{(m_{q,j})})+\sum_{\ell\in\NN_q^m}\tilde f_{q,\ell}(\tilde F_\ell^{(m_{q,\ell})}).
    \end{multline*}
    Clearly, the node before $q$ denoted by $q-1$ has only one path to $q$. Then, we can set to zero all the inputs except by $u_{q-1}^{k-m_{q,q-1}-1}$, to have for all $u_{q-1}^{k-m_{q,q-1}-1}\in\R$:
    $$
    f_{q,q-1}(u_{q-1}^{k-m_{q,q-1}-1})=\tilde f_{q,q-1}(u_{q-1}^{k-m_{q,q-1}-1}), 
    $$
    which implies $f_{q,q-1}=\tilde f_{q,q-1}$. Then, by Lemma~\ref{lemma:auxiliary}, we guarantee that $F_{q-1}=\tilde F_{q-1}$, which is equivalent to the measurement of the node $q-1$. Now, the identifiability problem becomes:

\vspace{-3mm}

\small    
    \begin{multline*}
        \!\!\!\sum_{j\in\NN_q^e\setminus \{q-1\}}f_{q,j}(u_j^{k-m_{q,j}-1}+F_j^{(m_{q,j})})+\sum_{\ell\in\NN_q^m}f_{q,\ell}(F_\ell^{(m_{q,\ell})})=\\
        \sum_{j\in\NN_q^e\setminus \{q-1\}}\tilde f_{q,j}(u_j^{k-m_{q,j}-1}+\tilde F_j^{(m_{q,j})})+\sum_{\ell\in\NN_q^m}\tilde f_{q,\ell}(\tilde F_\ell^{(m_{q,\ell})}).
    \end{multline*}
    \normalsize

\noindent    By Lemma~\ref{lemma:removal_node}, this new identifiability problem can be solved if the induced subgraph $G_{V\setminus\{ q-1\}}$ obtained with the removal of $q-1$ is identifiable. In this subgraph $G_{V\setminus\{ q-1\}}$, the node $q-2$ has only one path to $q$, and we can use a similar approach to prove that $f_{q,q-2}=\tilde f_{q,q-2}$ and $F_{q-2}=\tilde F_{q-2}$. By following the same procedure, we can guarantee that $F_v=\tilde F_v$ for all $v$ between $p$ and $q$, which corresponds to the mathematical constraints obtained with the measurement of these nodes.  
    In the case of the first measured node in the topological ordering, we apply the procedure until the first node in the topological ordering. Notice that this procedure can be used for any measured node such that all the nodes of the network can be covered. This implies that if the DAG is identifiable with $(\NN^e,V)$, it is also identifiable with $(\NN^e,\NN^m)$.
\end{myproof}

Since any identification pattern guarantees the access to the information corresponding to the measurement of all the nodes in the network, the measurement of an excited node does not change the identifiability of a DAG.

\begin{remark}[Patterns with $N^{m,e}>n$]
In an identification pattern $(\NN^e,\NN^m)$ that satisfies $N^{m,e}>n$, there must be at least a node which is excited and measured at the same time. If a DAG is identifiable with $(\NN^e,\NN^m)$, it is also identifiable with another identification pattern $(\hat\NN^e,\hat\NN^m)$ given by $\hat\NN^e= \NN^e$ and $\hat\NN^m=\NN^m\setminus\{\ell\}$, where $\ell\in\NN^e$ and $\ell\in\NN^m$. In this way, the identification pattern $(\NN^e,\NN^m)$ can generate a family of subpatterns by removing excited nodes from the set of measured nodes. 
\end{remark}

Proposition~\ref{prop:full_measurement_equivalence} illustrates that the information that can be obtained with the excitation of a node (extra input) is in some way more useful than the information that can be obtained with the measurement of a node.  
This notion will be also corroborated with the absence of symmetry full excitation/full measurement in the case of general DAGs.

\section{Identifiability conditions for partial excitation and measurement}

\subsection{Identifiability of trees}

We start by deriving identifiability conditions for trees in the case of partial excitation and measurement.

\begin{proposition}[Trees]\label{prop:trees_partial}
    A tree is identifiable in the class $\Ftwo$ with $\Nme=n$, if and only if all the sources are excited, all the sinks are measured and the other nodes are either excited or measured. 
\end{proposition}
\begin{proof}
By Lemma~\ref{lemma:sinks_sources}, the excitation of all the sources and the measurement of all the sinks is necessary, and by Proposition~\ref{prop:excite_measure}, the other nodes must be either excited or measured. We will prove the sufficiency by induction. According to Proposition~\ref{prop:full_measurement_equivalence}, the DAG is identifiable if (and only if) it is identifiable with $\NN^m=V$.
Let us assume that there exists a set $\{ f \}=\{ \tilde f \}$ such that $F(\NN^m)=\tilde F(\NN^m)$.
    Let us consider an arbitrary path of the tree that begins in a source and ends in a sink, and let us set to zero all the inputs except by the excitation signal corresponding to the source (node 1). For any node $i$ in the path that is not a source, the output is given by:
    \begin{align*}
        y_i^k
        &=f_{i,i-1}(\Theta_{i-1}(u_1^{k-M_{i,1}})),
    \end{align*}
    where $u_1$ is the excitation signal of the source with the corresponding delay $M_{i,1}$, and $\Theta_{i-1}$ is a composition of nonlinear functions associated with the edges. Let us assume that all the edges until the node $i-1$ have been identified. This implies that we know the function $\Theta_{i-1}$. Then, since $F_i\in F(\NN^m)$, we have the identifiability problem:
    $$
f_{i,i-1}(\Theta_{i-1}(u_1^{k-M_{i,1}}))=\tilde f_{i,i-1}(\Theta_{i-1}(u_1^{k-M_{i,1}})).
    $$
    Since the range of all the nonlinear functions is $\R$, we have that the range of $\Theta_{i-1}$ is also $\R$. Therefore, we obtain that $f_{i,i-1}(x)=\tilde f_{i,i-1}(x)$ for all $x\in\R$, which implies that $f_{i,i-1}=\tilde f_{i,i-1}$. 
    
    \noindent Now, notice that when $i=2$, the function $\Theta_{1}$ is the identity function, and clearly the edge $f_{2,1}$ can be identified. Then, by induction, all the edges of the path can be identified.
Finally, by considering other paths, we can identify all the tree.
\end{proof}

This shows that similarly to the linear case \cite{bazanella2019network}, any identification pattern with $N^{m,e}=n$, guarantees the identifiability of a tree. Notice that this also holds for path graphs and arborescences, which, unlike trees, necessarily have only one source.

According to Proposition~\ref{prop:trees_partial}, an identification pattern where all the sources are excited and the other nodes are measured, provides the identification of a tree. This corresponds to the symmetric result of the identifiability conditions for trees in the full excitation case where the measurement of the sinks is necessary and sufficient \cite{vizuete2023nonlinear,vizuete2024nonlinear}. Therefore, symmetry holds for trees.

\subsection{Generic nonlinear network matrix}

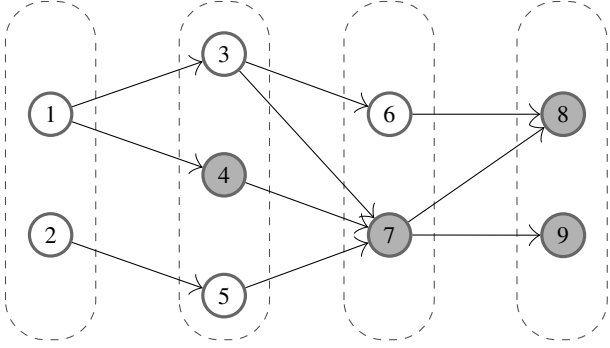
\begin{figure}
    \centering

\vspace{2mm}
    
    \begin{tikzpicture}
    [
roundnodes/.style={circle, draw=black!60, fill=black!5, very thick, minimum size=1mm},roundnode/.style={circle, draw=white!60, fill=white!5, very thick, minimum size=1mm},roundnodes2/.style={circle, draw=black!60, fill=black!30, very thick, minimum size=1mm},roundnodes3/.style={circle, draw=black!60, fill=white!30, very thick, minimum size=1mm}
]

\small

\node[roundnodes3](node1){1};
\node[roundnodes3](node2)[below=of node1,yshift=0mm,xshift=0cm]{2};

\node[roundnodes3](node3)[right=of node1,yshift=8mm,xshift=7mm]{3};
\node[roundnodes2](node4)[below=of node3,yshift=0mm,xshift=0cm]{4};
\node[roundnodes3](node5)[below=of node4,yshift=0mm,xshift=0cm]{5};

\node[roundnodes3](node6)[right=of node1,yshift=0mm,xshift=29mm]{6};
\node[roundnodes2](node7)[below=of node6,yshift=0mm,xshift=0cm]{7};

\node[roundnodes2](node8)[right=of node6,yshift=0mm,xshift=7mm]{8};
\node[roundnodes2](node9)[below=of node8,yshift=0mm,xshift=0cm]{9};

\draw[black!70,dashed,rounded corners=15pt]
  (-0.6,-3) rectangle ++(1.2,4.5) (0,0);
\draw[black!70,dashed,rounded corners=15pt]
  (1.71,-3) rectangle ++(1.2,4.5) (0,0);
  \draw[black!70,dashed,rounded corners=15pt]
  (3.9,-3) rectangle ++(1.2,4.5) (0,0);
  \draw[black!70,dashed,rounded corners=15pt]
  (6.2,-3) rectangle ++(1.2,4.5) (0,0);

\draw[-{Classical TikZ Rightarrow[length=1.5mm]}] (node1) to node [above,swap,yshift=0mm] {} (node3);
\draw[-{Classical TikZ Rightarrow[length=1.5mm]}] (node1) to node [above,swap,yshift=0mm] {} (node4);
\draw[-{Classical TikZ Rightarrow[length=1.5mm]}] (node2) to node [above,swap,yshift=0mm] {} (node5);
\draw[-{Classical TikZ Rightarrow[length=1.5mm]}] (node4) to node [above,swap,yshift=0mm] {} (node7);
\draw[-{Classical TikZ Rightarrow[length=1.5mm]}] (node5) to node [above,swap,yshift=0mm] {} (node7);
\draw[-{Classical TikZ Rightarrow[length=1.5mm]}] (node3) to node [above,swap,yshift=0mm] {} (node6);
\draw[-{Classical TikZ Rightarrow[length=1.5mm]}] (node3) to node [above,swap,yshift=0mm] {} (node7);
\draw[-{Classical TikZ Rightarrow[length=1.5mm]}] (node6) to node [above,swap,yshift=0mm] {} (node8);
\draw[-{Classical TikZ Rightarrow[length=1.5mm]}] (node7) to node [above,swap,yshift=0mm] {} (node9);
\draw[-{Classical TikZ Rightarrow[length=1.5mm]}] (node7) to node [above,swap,yshift=0mm] {} (node8);

\normalsize

\end{tikzpicture}
\vspace{-1mm}
    \caption{If the delays of all the nonlinear functions in a multipartite digraph are the same, all the excitation signals associated with a node in the function \eqref{eq:function_Fi} have the same delay.
    }

    \vspace{-4mm}
    \label{fig:multipartite_digraph}
\end{figure}

One of the main challenges in the identification of networks is to distinguish the information that arrives to a node $i$ through different paths. For instance, let us consider again the graph in Fig.~\ref{fig:model} with functions $f_{2,1}(y_1^{k-m_{2,1}})$, $f_{3,2}(y_2^{k-m_{3,2}})$ and $f_{3,1}(y_1^{k-m_{3,1}})$. 
The measurement of node 3 provides the output:
\begin{align}
y_3^k&=f_{3,1}(y_1^{k-m_{3,1}})+f_{3,2}(y_2^{k-m_{3,2}})\nonumber\\
&=f_{3,1}(u_1^{k-m_{3,1}-1})+f_{3,2}(f_{2,1}(u_1^{k-m_{2,1}-m_{3,2}-1})),\label{eq:relation_m_T_0}
\end{align}
If $m_{3,1}=m_{2,1}+m_{3,2}$, the output \eqref{eq:relation_m_T_0} becomes:
\begin{equation}\label{eq:relation_m_T}
    y_3^k=f_{3,1}(u_1^{k-m_{2,1}-m_{3,2}-1})+f_{3,2}(f_{2,1}(u_1^{k-m_{2,1}-m_{3,2}-1}))
\end{equation}
which implies that the information coming through $f_{3,1}$ and $f_{3,2}$ cannot be distinguished for this particular choice of delays $m_{i,j}$. 
Notice that for many networks,
there are particular choices of the delays associated with the nonlinear functions such that for each node $i$ in the network, the excitation signals corresponding to the same excited node in \eqref{eq:function_Fi} have the same delays.
For instance, this is the case for multipartite digraphs like the one in Fig.~\ref{fig:multipartite_digraph}, when the delays of all the nonlinear functions are the same. This particular setting is important for the generalization of results to more general dynamical models where the nonlinear functions can depend on the outputs of nodes with several delays and the mixing of information is possible \cite{vizuete2024nonlinear}. Furthermore, the identification of networks with this particular choice of delays is more challenging since the number of excitation variables that can be used is smaller. For this reason, we will analyze this worst-case scenario and we will consider that \eqref{eq:function_Fi} is given by:
\begin{equation}
    \label{eq:function_Fi_2}
    y_i^k\!=\!u_i^{k-1}\!+\!F_i(u_1^{k-T_{i,1}}\!,\ldots,u_{n_i}^{k-T_{i,n_i}}), \text{ for }\\
    1,\dots,n_i\!\in\! \mathcal{N}^{e\to i}.
\end{equation}
Regarding the measurement of the node 3 in \eqref{eq:relation_m_T}, the function $F_3$ will be of the form:
$$
y_3^k=F_3(u_1^{k-T_{3,1}}),
$$
where $T_{3,1}=m_{2,1}+m_{3,2}+1$.

In a more general DAG, a node can have more than one in-neighbor that shares several excitation signals.
Unlike the full excitation case, where each node has its own excitation signal, and identifiability conditions are valid for any function in a specific class, in the case of partial excitation and measurement, particular choices of functions could make a network unidentifiable.

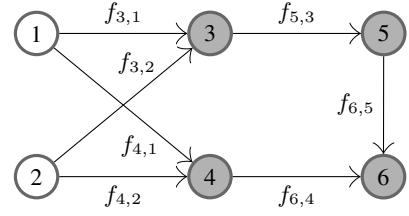
\begin{figure}[!t]
    \centering

\vspace{2mm}
    
    \begin{tikzpicture}
    [
roundnodes/.style={circle, draw=black!60, fill=black!5, very thick, minimum size=1mm},roundnode/.style={circle, draw=white!60, fill=white!5, very thick, minimum size=1mm},roundnodes2/.style={circle, draw=black!60, fill=black!30, very thick, minimum size=1mm},roundnodes3/.style={circle, draw=black!60, fill=white!30, very thick, minimum size=1mm}
]

\small

\node[roundnodes3](node1){1};
\node[roundnodes3](node5)[below=of node1,yshift=-0.3cm,xshift=0cm]{2};
\node[roundnodes2](node2)[right=of node1,yshift=0mm,xshift=0.7cm]{3};
\node[roundnodes2](node3)[right=of node5,yshift=0mm,xshift=0.7cm]{4};
\node[roundnodes2](node4)[right=of node3,yshift=0cm,xshift=0.7cm]{6};
\node[roundnodes2](node6)[right=of node2,yshift=0cm,xshift=0.7cm]{5};

\draw[-{Classical TikZ Rightarrow[length=1.5mm]}] (node1) to node [above,swap,yshift=0mm] {$f_{3,1}$} (node2);
\draw[-{Classical TikZ Rightarrow[length=1.5mm]}] (node2) to node [above,swap,yshift=0mm] {$f_{5,3}$} (node6);
\draw[-{Classical TikZ Rightarrow[length=1.5mm]}] (node1) to node [right,swap,xshift=-0.15cm,yshift=-5.5mm] {$f_{4,1}$} (node3);
\draw[-{Classical TikZ Rightarrow[length=1.5mm]}] (node3) to node [below,swap,yshift=0mm] {$f_{6,4}$} (node4);
\draw[-{Classical TikZ Rightarrow[length=1.5mm]}] (node5) to node [left,swap,xshift=0.55cm,yshift=5.5mm] {$f_{3,2}$} (node2);
\draw[-{Classical TikZ Rightarrow[length=1.5mm]}] (node5) to node [below,swap,yshift=0mm] {$f_{4,2}$} (node3);
\draw[-{Classical TikZ Rightarrow[length=1.5mm]}] (node6) to node [left,swap,yshift=0mm] {$f_{6,5}$} (node4);

\normalsize

\end{tikzpicture}
\vspace{-3mm}
    \caption{A DAG where a particular choice of functions makes the network unidentifiable. If $f_{3,1}=\gamma f_{4,1}$ and $f_{3,2}=\gamma f_{4,2}$ with $\gamma\neq 0$, the functions $f_{6,4}$ and $f_{6,5}$ cannot be identified.
    }
    \vspace{-4mm}
    \label{fig:DAG_unidentifiable}
\end{figure}

\begin{example}\label{ex:particular_choice}
    Let us consider the DAG in Fig.~\ref{fig:DAG_unidentifiable}. Let us assume that the functions $f_{3,1}$, $f_{3,2}$, $f_{4,1}$, $f_{4,2}$ and $f_{5,3}$ have been identified.
    The measurement of the node 6 provides the output:
    \begin{multline}\label{eq:example_generic}
    y_6^k=f_{6,5}(f_{5,3}(f_{3,1}(u_1^{k-T_{5,1}})+f_{3,2}(u_2^{k-T_{5,2}})))\\+f_{6,4}(f_{4,1}(u_1^{k-T_{5,1}})+f_{4,2}(u_2^{k-T_{5,2}})).    
    \end{multline}
    Notice that if $f_{3,1}=\gamma f_{4,1}$ and $f_{3,2}=\gamma f_{4,2}$ with $\gamma\neq 0$, the functions $f_{6,5}$ and $f_{6,4}$ cannot be identified since $\tilde f_{6,5}(x)=f_{6,4}(f_{5,3}^{-1}(\frac{x}{\gamma}))$ and $\tilde f_{6,4}(x)=f_{6,5}(f_{5,3}(\gamma x))$ also generates the output \eqref{eq:example_generic}.
\end{example}

Example~\ref{ex:particular_choice} shows that a DAG could be identifiable for most of the functions, except for particular choices
(e.g., two in-neighbors $p$ and $q$ of a node with outputs of the form $y_p=\gamma y_q$ for $\gamma\neq 0$).
This is similar to the concept of \emph{generic identifiability} in the linear case, where a network can be identifiable except for a particular choice of transfer functions that remains in a set of measure zero \cite{hendrickx2019identifiability}. In the context of nonlinear identifiability, we will introduce a  notion similar to the linear case by defining a generic matrix.

For a network $G$, let us define the matrix $J_\FF\in\R^{n\times n}$, where $[J_\FF]_{i,j}=f'_{i,j}(y_j^{k-m_{i,j}})$. This matrix can be considered as the adjacency matrix of a linear network $G_\FF$, that preserves the same topology of $G$ but with edges of the form $f'_{i,j}(y_j^{k-m_{i,j}})$. Since the output of any node can be expressed as a function of the excitation signals (see \eqref{eq:function_Fi_2}), the matrix $J_\FF$ can also be expressed as a function of the excitation signals with the corresponding delays. Now, let us consider a modification of the matrix $J_\FF$ where all the excitation signals corresponding to a node $i$, including different delays, are considered as the same input $v_i$. We call this new matrix the \emph{nonlinear network matrix} and we denote it by $J_G(\vv)$, where $\vv\in\R^{\abs{\NN^e}}$.
Since $J_G(\vv)$ is upper triangular in a DAG, we can guarantee that the matrix $T_G(\vv):=(I-J_G(\vv))^{-1}=\sum_{n=0}^\infty (J_G(\vv))^n$ is well \linebreak defined \cite{hendrickx2019identifiability}. Notice that an entry $[T_G]_{j,i}$ corresponds to the sum of the product of the edges corresponding to all the walks from $i$ to $j$ in $G_\FF$.

If we set free variables for the nonzero entries (i.e., edges) of the matrix $J_G(\vv)$, we say that for $A,B\subseteq V$, the rank of a submatrix $T_G^{A,B}(\vv)$ of  $T_G(\vv)=(I-J_G(\vv))^{-1}$ is maximal if it is maximal with respect to the free variables\linebreak of  $J_G(\vv)$ \cite{van1991graph}.

\begin{definition}[Generic nonlinear network matrix]
    We say that the nonlinear network matrix $J_G(\vv)$ associated to the network is generic if there exists a point $\vv^*\in\R^{\abs{\NN^e}}$ such that any submatrix $T_G^{A,B}(\vv^*)$ of $T_G(\vv^*)=(I-J_G(\vv^*))^{-1}$ has maximal rank.
\end{definition}
This generic nonlinear network matrix will be used in the derivation of identifiability conditions for general DAGs.

\subsection{Identifiability of DAGs}

Unlike trees, first, we will show that for more general DAGs, the identifiability conditions in the full excitation case do not have a symmetric equivalence in the full measurement case since the excitation of sources is not sufficient to guarantee identifiability.    

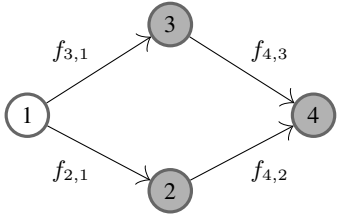
\begin{figure}
    \centering

\vspace{2mm}
    
    \begin{tikzpicture}
    [
roundnodes/.style={circle, draw=black!60, fill=black!5, very thick, minimum size=1mm},roundnode/.style={circle, draw=white!60, fill=white!5, very thick, minimum size=1mm},roundnodes2/.style={circle, draw=black!60, fill=black!30, very thick, minimum size=1mm},roundnodes3/.style={circle, draw=black!60, fill=white!30, very thick, minimum size=1mm}
]

\small

\node[roundnodes3](node1){1};
\node[roundnodes2](node2)[above=of node1,yshift=-4mm,xshift=1.9cm]{3};
\node[roundnodes2](node4)[right=3.2cm of node1]{4};
\node[roundnodes2](node3)[below=of node1,yshift=6mm,xshift=1.9cm]{2};

\draw[-{Classical TikZ Rightarrow[length=1.5mm]}] (node1) to node [left,swap,yshift=2.5mm] {$f_{3,1}$} (node2);
\draw[-{Classical TikZ Rightarrow[length=1.5mm]}] (node2) to node [right,swap,yshift=2.5mm] {$f_{4,3}$} (node4);
\draw[-{Classical TikZ Rightarrow[length=1.5mm]}] (node1) to node [left,swap,yshift=-2.5mm] {$f_{2,1}$} (node3);
\draw[-{Classical TikZ Rightarrow[length=1.5mm]}] (node3) to node [right,swap,yshift=-2.5mm] {$f_{4,2}$} (node4);

\normalsize

\end{tikzpicture}
\vspace{-2mm}
    \caption{Symmetry full excitation/full measurement does not hold for this DAG.
    }
    \vspace{-4mm}
    \label{fig:DAG_bridge}
    
\end{figure}

\begin{example}\label{ex:counterexample}
    Let us consider the DAG in Fig.~\ref{fig:DAG_bridge} with the functions $f_{2,1}(x)=a_{2,1}x^3$, $f_{3,1}(x)=a_{3,1}x^3$, $f_{4,2}(x)=a_{4,2}x^3$ and $f_{4,3}(x)=a_{4,3}x^3$, that belong to $\Ftwo$. If the symmetry full excitation/full measurement holds, the excitation of the source 1 should be sufficient to guarantee the identifiability of the network in the full measurement case. 
    The measurement of the node 2 provides the identification of $f_{2,1}$ and the measurement of the node 3 provides the identification of $f_{3,1}$. The measurement of the node 4 provides the output:   
    \begin{align}
    y_4^k
    &=f_{4,2}(f_{2,1}(u_1^{k-T_{4,1}}))+f_{4,3}(f_{3,1}(u_1^{k-T_{4,1}}))\nonumber\\
    &=(a_{4,2}a_{2,1}^3+a_{4,3}a_{3,1}^3)(u_1^{k-T_{4,1}})^9.\label{eq:F4}
    \end{align}
    Notice that the output \eqref{eq:F4} can also be generated with the functions $\tilde f_{4,2}(x)=(a_{4,2}+\gamma\frac{a_{3,1}^3}{a_{2,1}^3})x^3$ and $\tilde f_{4,3}(x)=(a_{4,3}-\gamma)x^3$ with $\gamma\neq 0$, which implies that it is not possible to identify the functions $f_{4,2}$ and $f_{4,3}$. 
 \end{example}

Unlike the DAG in Fig.~\ref{fig:DAG_unidentifiable}, the DAG in Fig.~\ref{fig:DAG_bridge}  is unidentifiable for any choice of the parameters of the cubic functions, which shows that the excitation of the sources is not sufficient to guarantee identifiability of DAGs in the class $\Ftwo$ in the full measurement case. In order to discard cases where the identifiability is not possible only due to a particular choice of functions, we make the following assumption.

\begin{assumption}\label{ass:nogamma}
    The nonlinear network matrix $J_G(\vv)$ associated with the network is generic.
\end{assumption}

Based on the notions of vertex-disjoint paths introduced in the linear case, we will provide sufficient conditions that guarantee identifiability of a DAG with a specific identification pattern.

\begin{definition}[Vertex-Disjoint Paths \cite{hendrickx2019identifiability}]
    A group of paths are mutually vertex disjoint if no two paths of this group contain the same vertex.
\end{definition}

\begin{theorem}[DAGs]\label{thm:DAG}
    Under Assumption~\ref{ass:nogamma}, a DAG is identifiable in the class $\Ftwo$ with $\Nme=n$, if:
    \begin{enumerate}
        \item All the sources are excited, all the sinks are measured and the other nodes are either excited or measured.
        \item There are vertex-disjoint paths from excited nodes to the in-neighbors of each node.
    \end{enumerate}
\end{theorem}

Before presenting the proof of Theorem~\ref{thm:DAG}, we recall a technical result that will be used in the proof.

\begin{lemma}[Theorem 2.35 \cite{laczkovich2017real}]\label{lemma:mapping}
    Let $H\subset \R^p$ and let $f:H\to\R^q$, where $p\ge q$. If $f$ is continuously differentiable at the point $a\in \text{int} \;H$ and the linear mapping $f'(a):\R^p\to\R^q$ is surjective, then the range of $f$ contains a neighborhood of $f(a)$.
\end{lemma}

\begin{myproof}{Theorem~\ref{thm:DAG}}
We will prove it by induction.
Let us consider an arbitrary DAG and an arbitrary topological ordering of the nodes.
According to Proposition~\ref{prop:full_measurement_equivalence}, the DAG is identifiable if (and only if) it is identifiable with $\NN^m=V$.
Let us consider an arbitrary node $i$. Without loss of generality, let us denote the set of in-neighbors of the node $i$ as $\NN_i=\{1,\ldots,m\}$. If we set to zero the possible excitation signal $u_i$ of the node $i$, the measurement of $i$ is given by:
\begin{align*}
    y_i^k&=\sum_{j=1}^m f_{i,j}(\phi_{i,j}(u_1^{k-T_{i,1}},\ldots,u_{n_i}^{k-T_{i,n_i}}))\\
    &=F_i(\phi_{i,1},\ldots,\phi_{i,m}),
\end{align*}
    where we use the same notation as in \eqref{eq:function_Fi_2}. 
    Let us assume that there exists a set $\{ f \}=\{ \tilde f \}$ such that $F(\NN^m)=\tilde F(\NN^m)$. Since $F_{i}\in F(\NN^m)$, we have the identifiability problem:
\begin{equation}\label{eq:identifiability_problem_zero}
    F_i(\phi_{i,1},\ldots,\phi_{i,m})=\tilde F_i(\tilde\phi_{i,1},\ldots,\tilde\phi_{i,m}).
\end{equation}
    Let us assume that all the incoming edges of the nodes at the left of the node $i$ in the topological ordering have been identified. 
Then, the identifiability problem \eqref{eq:identifiability_problem_zero} becomes:
\begin{equation}\label{eq:identifiability_problem}
    F_i(\phi_{i,1},\ldots,\phi_{i,m})=\tilde F_i(\phi_{i,1},\ldots,\phi_{i,m}).
\end{equation}
Now, let us consider the mapping: 
\vspace{-2mm}
$$\Phi_i:\R^{n_i}\to\R^m$$
\small
\vspace{-5mm}
$$
\Phi_i(w_1,\ldots,w_{n_i})\!=\!(\phi_{i,1}(w_1,\ldots,w_{n_i}),\ldots,\phi_{i,m}(w_1,\ldots,w_{n_i})),
\vspace{-1mm}
$$
\normalsize
and the Jacobian matrix of $\Phi_i$ denoted by $J_{\Phi_i}(\ww)$ where $\ww=(w_1,\ldots,w_{n_i})$.
Notice that $J_{\Phi_i}(\ww)$ corresponds to the submatrix of the nonlinear network matrix $T_G(\vv)$ where the rows are the in-neighbors of $i$ and the columns are the nodes with the excitation signals. Under Assumption~\ref{ass:nogamma}, there exists a point $\vv^*$ such that $T_G(\vv^*)$ has maximal rank, which also determines a point $\ww^*$ for $J_{\Phi_i}(\ww^*)$.
According to \cite[Proposition~V.1]{hendrickx2019identifiability}, since there are vertex-disjoint paths from excited nodes to the in-neighbors of $i$, the generic rank of $J_{\Phi_i}(\ww^*)$ is $m$, which implies that $J_{\Phi_i}(\ww^*)$ is surjective.
By virtue of Lemma~\ref{lemma:mapping}, the range of $\Phi_i$ must contain a neighborhood of $\Phi_i(\ww^*)$, which implies that \eqref{eq:identifiability_problem} holds in a set of positive measure. Then, by the Identifiy Theorem of analytic functions \cite{krantz2002primer}, we guarantee that 
\vspace{-1mm}
\begin{equation}\label{eq:result_theorem}
F_i(z_1,\ldots,z_m)=\tilde F_i(z_1,\ldots,z_m) \;\text{ for all } z_1,\ldots,z_m\in\R,    
\end{equation}
since the range of all the functions $\phi_{i,j}$ is $\R$. From \eqref{eq:result_theorem}, it follows that
\vspace{-2mm}
$$
f_{i,j}=\tilde f_{i,j} \quad \text{for all } j=1\ldots,m,
\vspace{-1mm}
$$
so that all the incoming edges of $i$ can be identified.

\noindent Now, notice that for $i=2$ in the topological ordering, if there is an edge $f_{2,1}$, the function $\phi_{2,1}$ is the identity function and the edge $f_{2,1}$ can be clearly identified. Then, by induction, the identifiablity analysis is valid for any node in the DAG. Thus, all the DAG is identifiable.
\end{myproof}

Theorem~\ref{thm:DAG} implies that a DAG satisfies symmetry full excitation/full measurement if there are vertex-disjoint paths from the sources to the in-neighbors of each node of the network.
This is also consistent with the case of trees since from all the sources we have vertex disjoint paths that reach the in-neighbors of all the nodes of a tree.

When the functions are linear, the space of functions that do not satisfy Assumption~\ref{ass:nogamma} has a dimension smaller than the dimension of the system (i.e., set of measure zero) \cite{hendrickx2019identifiability}. In the case of nonlinear functions, to the best of our knowledge, the dimension of the space of functions that do not satisfy Assumption~\ref{ass:nogamma} is still not known and its full characterization is left for future work.

Theorem~\ref{thm:DAG} provides sufficient conditions for the identifiability of DAGs that are weaker than in the linear case. 
For instance, let us consider the DAG in Fig.~\ref{fig:DAG_unidentifiable_linear}. In the linear case, the DAG requires $N^{m,e}>n$ since the node 1 is a dource and a dink \cite{mapurunga2022excitation}, and must be measured and excited. However, in the nonlinear case, we only need $N^{m,e}=n$ to guarantee identifiability. The node 3 has a vertex-disjoint path from the node 1 and the node 4 has a vertex-disjoint path from 2, so that the DAG satisfies the conditions of Theorem~\ref{thm:DAG} and it is identifiable with this identification pattern.
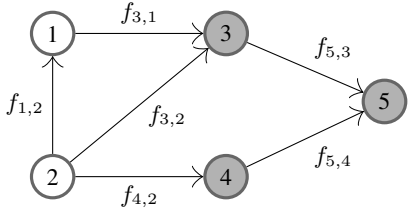
\begin{figure}[!t]
    \centering

\vspace{2mm}
    
    \begin{tikzpicture}
    [
roundnodes/.style={circle, draw=black!60, fill=black!5, very thick, minimum size=1mm},roundnode/.style={circle, draw=white!60, fill=white!5, very thick, minimum size=1mm},roundnodes2/.style={circle, draw=black!60, fill=black!30, very thick, minimum size=1mm},roundnodes3/.style={circle, draw=black!60, fill=white!30, very thick, minimum size=1mm}
]

\small

\node[roundnodes3](node1){1};
\node[roundnodes3](node5)[below=of node1,yshift=-0.3cm,xshift=0cm]{2};
\node[roundnodes2](node2)[right=of node1,yshift=0mm,xshift=0.7cm]{3};
\node[roundnodes2](node3)[right=of node5,yshift=0mm,xshift=0.7cm]{4};
\node[roundnodes2](node4)[right=of node3,yshift=1cm,xshift=0.5cm]{5};

\draw[-{Classical TikZ Rightarrow[length=1.5mm]}] (node1) to node [above,swap,yshift=0mm] {$f_{3,1}$} (node2);
\draw[-{Classical TikZ Rightarrow[length=1.5mm]}] (node2) to node [right,swap,yshift=2.5mm] {$f_{5,3}$} (node4);
\draw[-{Classical TikZ Rightarrow[length=1.5mm]}] (node5) to node [left,swap,xshift=0cm,yshift=0mm] {$f_{1,2}$} (node1);
\draw[-{Classical TikZ Rightarrow[length=1.5mm]}] (node3) to node [right,swap,yshift=-2.5mm] {$f_{5,4}$} (node4);
\draw[-{Classical TikZ Rightarrow[length=1.5mm]}] (node5) to node [right,swap,xshift=0.00cm,yshift=-1.5mm] {$f_{3,2}$} (node2);
\draw[-{Classical TikZ Rightarrow[length=1.5mm]}] (node5) to node [below,swap,yshift=0mm] {$f_{4,2}$} (node3);

\normalsize

\end{tikzpicture}
\vspace{-3mm}
    \caption{A DAG with an identification pattern that is unidentifiable in the linear case but it is identifiable in the nonlinear case.
    }
    \vspace{-4mm}
    \label{fig:DAG_unidentifiable_linear}
\end{figure}

\section{Conclusions}

We analyzed the identifiability of DAGs with partial excitation and measurement. We showed that in the nonlinear case, a DAG is identifiable with an identification pattern if and only if it is identifiable with the full measurement of the nodes. For trees, we showed that the symmetry full excitation/full measurement holds and that any identification pattern guarantees identifiability.
Unlike the linear case, in the case of more general DAGs, we showed that symmetry does not hold. Finally, we introduced the notion of a generic nonlinear network matrix, and we derived identifiability conditions based on vertex-disjoint paths from excited nodes to the in-neighbors of each node in the 
network.

For future work, it would be important to characterize the space of functions that do not satisfy Assumption~\ref{ass:nogamma}.
Also, it would be interesting to analyze the case when the range of the functions is not all $\R$. Finally, a generalization of the analysis to more general digraphs where cycles might exist ($F_i$ depends on an infinite number of inputs), and more general models where the nonlinear function can depend on inputs with several \linebreak delays \cite{vizuete2024nonlinear} is definitely an interesting area of research.

\bibliographystyle{IEEEtran}
\bibliography{arXiv}

\end{document}